\documentclass[12pt]{amsart}
\usepackage{latexsym,amssymb,amsmath}

\numberwithin{equation}{section}
\newtheorem{theorem}{Theorem}[section]
\newtheorem{lemma}[theorem]{Lemma}

\newtheorem{corollary}[theorem]{Corollary}

\theoremstyle{definition}

\newtheorem{example}[theorem]{Example}

\begin{document}

%%%%%%%%%%%%%%%%%%%%%%%%%%%%%%%%%%%%%%%%%%%%%%%%%%%%%%%%%%%%%%%%%%
% New Commands

\newcommand{\del}{\operatorname{del}}
\newcommand{\lk}{\operatorname{lk}}
\newcommand{\reg}{\operatorname{reg}}
\newcommand{\pd}{\operatorname{pd}}

%%%%%%%%%%%%%%%%%%%%%%%%%%%%%%%%%%%%%%%%%%%%%%%%%%%%%%%%%%%%%%%% 
 
\title[]
{A note on the van der Waerden complex}
\thanks{Submitted Version: July 24, 2017}

\author{Becky Hooper}
\address{Department of Mathematics and Statistics\\
McMaster University, Hamilton, ON, L8S 4L8}
\curraddr{759  Hillside Rd., Albert Bridge, NS, B1K 3H7}
\email{hooperb@mcmaster.ca,becky9997@hotmail.com}
 
\author{Adam Van Tuyl}
\address{Department of Mathematics and Statistics\\
McMaster University, Hamilton, ON, L8S 4L8}
\email{vantuyl@math.mcmaster.ca}
 
\keywords{van der Waerden complex, vertex decomposable, 
shellable, Cohen-Macaulay}
\subjclass[2000]{05E45,13F55}

\begin{abstract}
Ehrenborg, Govindaiah, Park, and Readdy recently introduced
the van der Waerden complex, 
a pure simplicial complex whose facets correspond to arithmetic
progressions.  Using techniques from combinatorial commutative
algebra, we classify when these pure simplicial complexes
are vertex decomposable or not Cohen-Macaulay.  As a corollary,
we classify the van der Waerden complexes that are shellable. 
\end{abstract}
 
\maketitle

%%%%%%%%%%%%%%%%%%%%%%%%%%%%%%%%%%%%%%%%%%%%%%%%%%%%%%%%%%%%%%%%%%%

\section{Introduction} 

Let $V = \{x_1,\ldots,x_n\}$ and suppose that $0 < k < n$.  The
{\it van der Waerden complex} of dimension $k$ on $n$
vertices, denoted ${\tt vdW}(n,k)$, is the pure simplicial
complex on $V$ whose facet set is given by 
\[{\tt vdW}(n,k) = 
\left\langle \{x_i,x_{i+d},x_{i+2d},\allowbreak \ldots,x_{i+kd}\}
~|~ d\in\mathbb{Z} ~~\mbox{with $1\leq i<i+kd\leq n$} \right\rangle.\]
In other words, the facets of ${\tt vdW}(n,k)$ correspond to
all arithmetic progressions of length $k+1$ whose largest element is
less than or equal to $n$.   The complexes ${\tt vdW}(n,k)$ 
were introduced by Ehrenborg, Govindaiah, Park, and Readdy \cite{EGPR}
as part of a recent program to study the topology of complexes that arise
within number theory.   In particular, the work of \cite{EGPR} focused on the
homotopy type of ${\tt vdW}(n,k)$.

The van der Waerden complex is a pure simplicial complex.
It is known that pure simplicial complexes may have additional
combinatorial and topological properties, e.g., vertex
decomposable, shellable, and Cohen-Macaulay.  Specifically, we have
the following chain of implications (definitions are postponed
until the next section):
\[
\mbox{vertex decomposable} \Rightarrow
\mbox{shellable} \Rightarrow
\mbox{Cohen-Macaulay} \Rightarrow \mbox{pure}.\]
In general, these implications are all strict.  It is natural
to ask when ${\tt vdW}(n,k)$ has these additional
properties in terms of $n$ and $k$.    
We answer this question in this note; precisely:

\begin{theorem}\label{maintheorem}
Let  $0<k<n$ be integers.   Then 
\begin{enumerate}
\item[$(i)$] ${\tt vdW}(n,k)$ is vertex decomposable if and only if 
\begin{enumerate}
\item[$\bullet$] $n \leq 6$, or 
\item[$\bullet$] $n > 6$ and $k=1$, or
\item[$\bullet$] $n> 6$ and $\frac{n}{2} \leq k <n$.
\end{enumerate}
\item[$(ii)$] ${\tt vdW}(n,k)$ is pure but not Cohen-Macaulay  
if and only if $n > 6$ and $2 \leq k < \frac{n}{2}$.
\end{enumerate}
\end{theorem}

As a corollary, we can recover a result of \cite{H}
first proved using different techniques.

\begin{corollary}\label{maincor}
Let $0<k<n$ be integers.  Then ${\tt vdW}(n,k)$ is shellable if and only if
\begin{itemize}
\item $n\leq6$, or
\item $n>6$ and $k=1$, or
\item $n>6$ and $\frac{n}{2}\leq k<n$.
\end{itemize}
\end{corollary}

\begin{proof}
If $k$ and $n$ satisfy the above conditions, then 
${\tt vdW}(n,k)$ is vertex decomposable by Theorem \ref{maintheorem}, 
and consequently, shellable.  Otherwise
${\tt vdW}(n,k)$ is not Cohen-Macaulay by Theorem \ref{maintheorem}, 
so it cannot
be shellable.
\end{proof}

Our paper is structured as follows.  We first
recall the relevant background in Section 2.  In Section 3
we prove Theorem \ref{maintheorem} 
using some tools from combinatorial commutative algebra.
In particular, to show that ${\tt vdW}(n,k)$ is not
Cohen-Macaulay, we will 
show that the Stanley-Reisner ideal of the Alexander dual of 
${\tt vdW}(n,k)$ has nonlinear first syzygies.

\noindent
{\bf Acknowledgments.} 
Parts of this paper appeared in the first author's M.Sc. project \cite{H}.  
The second author acknowledges the financial support of NSERC.

%%%%%%%%%%%%%%%%%%%%%%%%%%%%%%%%%%%%%%%%%%%%%%%%%%%%%%%%%%%%%%%%%%

\section{Background}

In this section we recall the relevant combinatorial
and algebraic background.

Let $V = \{x_1,\ldots,x_n\}$ be a vertex set.  A {\it simplicial complex}
on $V$ is a subset $\Delta \subseteq 2^{V}$ such that $(a)$ if $F \in
\Delta$ and $G \subseteq F$, then $G \in \Delta$, and $(b)$ $\{x_i\} \in \Delta$
for all $i \in \{1,\ldots,n\}$.  Elements of $\Delta$ are called
{\it faces}, and maximal faces under inclusion are called {\it facets}.
If $F_1,\ldots,F_s$ is a complete list of facets of $\Delta$, we usually
write $\Delta = \langle F_1,\ldots,F_s \rangle$.   The {\it dimension}
of a face $F$, denoted $\dim(F)$, is $\dim(F) = |F|-1$.  The {\it dimension
of $\Delta$}, denoted $\dim \Delta$, is $\dim \Delta = \max\{\dim(F) ~|~
\mbox{$F$ a facet of $\Delta$}\}$.  A simplicial complex is 
{\it pure} if all its facets have the same dimension.

The {\it Alexander dual} of $\Delta$, denoted $\Delta^\vee$, is 
the simplicial complex whose facets are complements of the minimal 
non-faces of $\Delta$. That is,
$
\Delta^\vee = \{ V\backslash F ~|~  F\notin\Delta \}.
$

To any simplicial complex $\Delta$, the {\it Stanley-Reisner ideal}
of $\Delta$ is a monomial ideal $I_\Delta$ in the polynomial
ring $R = k[x_1,\ldots,x_n]$ where 
\[I_{\Delta} = \langle x_{i_1}x_{i_2}\cdots x_{i_t} ~|~ \{x_{i_1},\ldots,x_{i_t}\} 
\not\in \Delta \rangle.\]
The following result allows us to directly write out the minimal generators
of the Stanley-Reisner ideal of the Alexander dual of $\Delta$
from the facets of $\Delta$.

\begin{lemma}[{\cite[Corollary 1.5.5]{HH}}]\label{Idv}
Let $\Delta=\langle F_1,F_2\ldots,F_s \rangle$.  Then
\[ I_{\Delta^\vee} = \left\langle m_{F_1^c},\ldots,m_{F_s^c} \right\rangle 
~~\mbox{where $m_{F_i^c} = \displaystyle\prod_{x\notin F_i}x$.}\] 
\end{lemma}

We recall three families of pure simplicial complexes. 
The first family was introduced by Provan and Billera \cite{PB};
a pure simplicial complex $\Delta$ on $V$ is {\it vertex decomposable}
if 
\begin{enumerate}
\item[$(i)$] $\Delta = \emptyset$, or $\Delta = 
\langle \{x_1,\ldots,x_n\} \rangle,$ i.e., a simplex; or
\item[$(ii)$] there exists a vertex $x \in V$ such that
the {\it link}
of $x$, i.e.,
\[\lk_{\Delta}(x) = \{ H \in \Delta ~|~ H \cap \{x\} = \emptyset ~\mbox{and}~~
H \cup \{x\} \in \Delta\},\]
and the {\it deletion} of $x$, i.e.,
$\del_{\Delta}(x) = \{H \in \Delta ~|~ H \cap \{x\} = \emptyset \}$,
are both vertex decomposable simplicial complexes.
\end{enumerate}

The second family is the family of shellable
simplicial complexes.  A pure complex $\Delta$ is
{\it shellable} if 
the facets of $\Delta$ can be ordered, say  $F_1,\ldots,F_s$, 
such that for all $1\leq i<j\leq s$, there 
exists some $x\in F_j\setminus F_i$ and some 
$\ell\in \{1,\ldots,j-1\}$ with $F_j\setminus F_\ell= \{x\}$. 

Finally, a pure simplicial complex $\Delta$ is {\it
Cohen-Macaulay\footnote[1]{One normally defines a simplicial complex
$\Delta$
to be Cohen-Macaulay either in terms of the depth and dimension
of $R/I_{\Delta}$, or in terms of the reduced simplicial homology
of $\Delta$. Our  definition uses the 
characterization of Cohen-Macaulay simplicial
complexes due to Eagon and Reiner \cite{ER}.} over $k$} 
if the minimal free resolution of $I_{\Delta^\vee}$
over $R =k[x_1,\ldots,x_n]$
is linear.  Recall that an ideal $I \subseteq R = k[x_1,\ldots,x_n]$
has a {\it linear minimal free resolution} if $I$ has a minimal free
resolution of the form
\[0 \rightarrow 
R^{b_t}(-d-t)                        
\longrightarrow
\cdots
\longrightarrow
R^{b_2}(-d-2)                        
\longrightarrow
R^{b_1}(-d-1)                        
\longrightarrow
R^{b_0}(-d)                        
\longrightarrow
I
\longrightarrow
0
\]
for some integer $d$
where $R(-d-i)$ denotes the polynomial ring shifted by degree $d+i$
and $R^{b_i}(-d-i) = R(-d-i) \oplus \cdots \oplus R(-d-i)$ ($b_i$ times).

We now state some of the basic results that we require, with references
to their proofs.

\begin{theorem}\label{vdprop}
Let $\Delta$ be a pure simplicial complex.
\begin{enumerate}
\item[$(i)$] If $\Delta$ is vertex decomposable, then $\Delta$ is shellable.
\item[$(ii)$] If $\Delta$ is shellable, then $\Delta$ is Cohen-Macaulay.
\item[$(iii)$] If $\dim \Delta =1$ and $\Delta$ is connected,
then $\Delta$ is vertex decomposable.
\end{enumerate}
\end{theorem}

\begin{proof} $(i)$ is \cite[Corollary 2.9]{PB}; $(ii)$ is 
\cite[Theorem 5.3.18]{V}; and $(iii)$ is \cite[Theorem 3.1.2]{PB}.
\end{proof}

\begin{example}\label{specialcases}
We show that
both ${\tt vdW}(5,2)$ and ${\tt vdW}(6,2)$ are vertex decomposable.
Not only do these examples illuminate our definitions,
we require these special arguments for these complexes to prove Theorem 
\ref{maintheorem}.
We begin with 
\[\Delta = {\tt vdW}(5,2) = 
\langle \{x_1,x_2,x_3\},\{x_2,x_3,x_4\},\{x_3,x_4,x_5\},\{x_1,x_3,x_5\} 
\rangle. \]
We form the deletion and link of $x_5$:
\begin{eqnarray*}
\del_\Delta(x_5) &=&
 \langle \{x_1,x_2,x_3\},\{x_2,x_3,x_4\} \rangle ~~\mbox{and}~~ \lk_\Delta(x_5) = \langle \{x_3,x_4\},\{x_1,x_3\}\rangle.
\end{eqnarray*}
Now  $\lk_\Delta(x_5)$ is vertex decomposable by Theorem \ref{vdprop} $(iii)$.
Let $\Gamma = \del_\Delta(x_5)$ and form the link and deletion
with respect to $x_4$:
\[\del_\Gamma(x_4) =
 \langle \{x_1,x_2,x_3\} \rangle ~~\mbox{and}~~ \lk_\Gamma(x_5) = 
\langle \{x_2,x_3 \} \rangle.\]
Both of these complexes are simplicies, so $\del_\Delta(x_5)$ is 
vertex decomposable, and consequently, so is ${\tt vdW}(5,2)$

The proof for the complex
\[\Delta = {\tt vdW}(6,2) = \langle \{x_1,x_2,x_3\},\{x_2,x_3,x_4\},\{x_3,x_4,x_5\},
\{x_4,x_5,x_6\},\{x_1,x_3,x_5\},\{x_2,x_4,x_6\} \rangle\]
is similar.  We form the deletion and link of $x_6$.  In particular,
\begin{eqnarray*}
\del_\Delta(x_6) &=&
 \langle \{x_1,x_2,x_3\},\{x_2,x_3,x_4\},\{x_3,x_4,x_5\},\{x_1,x_3,x_5\}  \rangle
= {\tt vdW}(5,2), ~~\mbox{and}~~ \\
\lk_\Delta(x_6)& =& \langle \{x_4,x_5\},\{x_2,x_4\}\rangle.
\end{eqnarray*}
We just showed that ${\tt vdW}(5,2) = \del_\Delta(x_6)$ 
is vertex decomposable, and  $\lk_\Delta(x_6)$ is vertex decomposable
by Theorem \ref{vdprop} $(iii)$.  So, ${\tt vdW}(6,2)$ is vertex decomposable.
\end{example}

We complete this section with some results about the first syzygy module 
of a monomial ideal.  Let $I$ be a monomial ideal of $R = k[x_1,\ldots,x_n]$
whose unique set of minimal generators are $G(I) = \{m_1,\ldots,m_s\}$.
Let $d_i = \deg(m_i)$ for $i=1,\ldots,s$, and let $e_{m_i}$ denote
the basis element of the shifted $R$-module $R(-d_i)$.
We can then construct the
following degree zero $R$-module homomorphism
\[\varphi: M = R(-d_1) \oplus R(-d_2) \oplus \cdots \oplus R(-d_s) \longrightarrow
I\]
where $e_{m_i} \mapsto m_i$   for $i=1,\ldots,s$.
The {\it first syzygy module of $I$} is then
\[{\rm Syz}^1_R(I) = \{(F_1,\ldots,F_s) \in  M ~|~ \varphi(F_1,\ldots,F_s)
= F_1m_1+\cdots + F_sm_s = 0\},\]
i.e., ${\rm Syz}^1_R(I) = {\rm ker}(\varphi)$.   
The module ${\rm Syz}^1_R(I)$ is a finitely generated $R$-module; in fact:

\begin{theorem}[{\cite[Corollary 4.13]{EH}}]\label{generatorsSYZ}
Let $I \subseteq R = k[x_1,\ldots,x_n]$ be
a monomial ideal with minimal generators $G(I) = \{m_1,\ldots,m_s\}$.
Then
\[{\rm Syz}^1_R(I) = \langle \sigma_{j,i}e_{m_i} - \sigma_{i,j}e_{m_j} ~|~
1 \leq i < j \leq s \rangle
~~\mbox{where}~~\sigma_{i,j} =  \frac{m_i}{{\rm gcd}(m_i,m_j)}.\] 
\end{theorem}

The set of generators in the above result may not be a 
minimal set of generators.   However, some subset of these
generators is a minimal set of generators.   
The first syzygy module is {\it generated by linear first syzygies} 
if there is some subset $T \subseteq 
\{ \sigma_{j,i}e_{m_i} - \sigma_{i,j}e_{m_j} ~|~
1 \leq i < j \leq s \}$ that generates ${\rm Syz}_R^1(I)$, 
and for all $\sigma_{j,i}e_{m_i} - \sigma_{i,j}e_{m_j} \in T$,
$\deg \sigma_{i,j} = \deg \sigma_{j,i} = 1$.    

The construction of ${\rm Syz}_R^1(I)$ is the first step
in the construction of the minimal free resolution of $I$.  In
particular, we have the following fact.

\begin{theorem}\label{firstSYZ}  
If $I$ is a monomial ideal with a linear resolution,
then ${\rm Syz}_R^1(I)$ is generated by linear first syzygies.
\end{theorem}

%%%%%%%%%%%%%%%%%%%%%%%%%%%%%%%%%%%%%%%%%%%%%%%%%%%%%%%%%%%%%%%%

\section{Proof of the main theorem}

We prove Theorem \ref{maintheorem} in this section.  To do
so,  we require the following two lemmas about the 
facets of ${\tt vdW}(n,k)$.   
Given  a facet $F=\{x_i,x_{i+d},x_{i+2d},\ldots,x_{i+kd}\} \in {\tt vdW}(n,k)$, 
we call $d$ the {\it increment of $F$}.  Note that every facet has an
associated increment.

\begin{lemma}\label{mainlemma1}
Suppose $n \geq 7$.   Let $F \in {\tt vdW}(n,2)$ be any facet
such that its increment is the largest possible odd integer $d$.
If $G \in {\tt vdW}(n,2)$ is any other facet with increment
$d' \neq d$, then $|F \cap G| \leq 1$.
\end{lemma}

\begin{proof}
Because $n \geq 7$, the complex ${\tt vdW}(n,2)$ contains the facet $\{1,4,7\}$.
Thus the largest odd increment $d$ satisfies $d \geq 3$.  
Let $F = \{x_a,x_{a+d},x_{a+2d}\}$ be any facet whose increment is $d$
and let $G = \{x_b,x_{b+d'},x_{b+2d'}\}$ be any other facet whose
increment is $d' \neq d$.    

It is immediate that $F \neq G$, so $|F \cap G| \leq 2$.   So
suppose $|F \cap G| = 2$.  Since $a < a+d < a+2d$ and $b < b+d'
< b+2d'$, we have the following possible cases:
\begin{center}
\begin{tabular}{llll}
$(a)$ & $a=b$ and $a+d = b+d'$  & $(f)$ & $a=b+d'$ and $a+2d = b+2d'$\\
$(b)$ & $a=b$ and $a+d = b+2d'$ & $(g)$ & $a+d=b$ and $a+2d = b+d'$ \\
$(c)$ & $a=b$ and $a+2d = b+d'$ & $(h)$ & $a+d=b$ and $a+2d = b+2d'$ \\
$(d)$ & $a=b$ and $a+2d = b+2d'$ & $(i)$ & $a+d=b+d'$ and $a+2d=b+2d'$.\\
$(e)$ & $a=b+d'$ and $a+d = b+2d'$ && \\
\end{tabular}
\end{center}
Cases $(a), (d), (e), (g)$ and $(i)$ all imply $d=d'$, so we can eliminate
those cases.  For cases $(b)$ and $(h)$,  we would have $d=2d'$, which implies
that the odd integer $d$ is even, so this case cannot happen.  
Finally, for cases $(c)$ and $(f)$, we would have 
$2d = d'$.  But $d\geq 3$ is the largest
odd increment, so the largest increment of ${\tt vdW}(n,2)$ is either
$d$ or $d+1$.  But $d' = 2d > d+1$, so this is not a valid
increment, and consequently, this case cannot happen. 

Therefore, it must be the case that $|F \cap G| \leq 1$.
\end{proof}

We now prove a similar lemma, but now we do not require the increment
to be odd.

\begin{lemma}\label{mainlemma2}
Suppose $n \geq 7$ and $2 < k <\frac{n}{2}$.  
Let $F \in {\tt vdW}(n,k)$ be any facet whose increment
$d$ is the largest possible.
If $G \in {\tt vdW}(n,k)$ is any other facet with increment
$d' \neq d$, then $|F \cap G| \leq k-1$.
\end{lemma}

\begin{proof}
Since $k < \frac{n}{2}$, we have $\{x_1,x_3,\ldots,x_{1+2k}\}
\in {\tt vdW}(n,k)$.   If $F \in {\tt vdW}(n,k)$ has the largest possible
increment $d$, we must therefore have $d \geq 2$.

Let $F= \{x_a,x_{a+d},\ldots,x_{a+kd}\}$ be a facet with
increment $d$, and suppose that the facet  $G = \{x_b,x_{b+d'},\ldots,x_{b+kd'}\}$
has increment $d' \neq d$.  Since the
facets are distinct, we must have $|F \cap G| \leq k$.   

Suppose that $|F \cap G| = k$.  Since $|G| = k \geq 3$, there
must be $x_{b+id'},x_{b+(i+1)d'} \in G$, i.e., two consecutive 
terms of the arithmetic progression in $G$ such that 
\[a+\ell d = b+id' ~~\mbox{and} ~~~ a+jd = b+(i+1)d' 
~~\mbox{for some $\ell < j$.}\]  
But these two equations imply that $(j-\ell)d = d'$, i.e., $d' \geq d$,
contradicting the fact that $d$ is the largest increment.  So
$|F \cap G| \leq k-1$.
\end{proof}

We now prove Theorem \ref{maintheorem}.

\begin{proof}{(of Theorem \ref{maintheorem})}
We break the proof into cases depending on $0 < k < n$.

{\bf Case 1:} $k=1$ and $1 <n$. In this case  ${\tt vdW}(n,1)$ is vertex decomposable
by Theorem \ref{vdprop} $(iii)$ because 
\[{\tt vdW}(n,1) = \langle \{x_i,x_j\} ~|~ 1 \leq i < j \leq n \rangle,\]
is a connected one-dimensional simplicial complex.

{\bf Case 2:} $\frac{n}{2} \leq k < n$.
If $1 = k < 2$, then
${\tt vdW}(2,1)$ is vertex decomposable by the previous case.
We now proceed by induction on $n$.  If $k = n-1$, then
${\tt vdW}(n,n-1) = \langle \{x_1,x_2,x_3,\ldots,x_n\} \rangle$ is a
simplex, and hence, vertex decomposable.  

So suppose that  $\frac{n}{2} \leq k < n-1$.
Every facet of  ${\tt vdW}(n,k)$ must have increment $d=1$
since $\frac{n}{2} \leq k$.  So
\[\Delta = {\tt vdW}(n,k) = \langle \{x_1,x_2,\ldots,x_{k+1}\},
\{x_2,x_3,\ldots,x_{k+2}\},
\ldots,\{x_{n-k},\ldots,x_n\} \rangle.\]
We form the link and deletion of $x_n$:
\begin{eqnarray*}
\del_\Delta(x_n) & = & {\tt vdW}(n-1,k) ~~\mbox{and}~~
\lk_\Delta(x_n) =  \langle \{x_{n-k},\ldots,x_{n-1}\}\rangle.
\end{eqnarray*}
Since $\frac{n-1}{2} < k < n-1$, by induction ${\tt vdW}(n-1,k)$ 
is vertex decomposable.  Because $\lk_\Delta(x_n)$ is a simplex,
we can now conclude that ${\tt vdW}(n,k)$ is vertex decomposable
if $\frac{n}{2} \leq k < n$.

{\bf Case 3:} $0 < k < n \leq 6$.
The only $n$ and $k$ in this case not covered by Case 1 or 2 is
$(n,k) = (5,2)$ or $(6,2)$.  We now use Example \ref{specialcases}
to complete this case.

{\bf Case 4:} $n > 6$ and $2 \leq k < \frac{n}{2}$. 
Let $I = I_{{\tt vdW}(n,k)^\vee}$ be the Stanley-Reisner ideal of the 
Alexander dual of ${\tt vdW}(n,k)$.  We will show that 
${\rm Syz}_R^1(I)$ cannot be generated by linear first syzygies.
It will then follow by Theorem \ref{firstSYZ} that $I$ does
not have a linear minimal free resolution, and consequently,
${\tt vdW}(n,k)$ is a  simplicial complex that is pure but
not Cohen-Macaulay.

If ${\tt vdW}(n,k) = \langle F_1,\ldots,F_s \rangle$, then
by Lemma \ref{Idv}, 
\[I = \left.\left\langle m_{F_i^c} = \prod_{x \not\in F_i} x ~\right|~ i=1,\ldots,s 
\right\rangle.\]
Since the complex is pure, this ideal is generated by $s$
monomials all of degree $n - k -1$.

We first consider the case that $3 \leq k < \frac{n}{2}$.   Let
$F$ be any facet with the largest increment $d$.  Since $n >6$, we know that
$d \geq 3$.  Now take another facet $G$ with increment $d' \neq d$.
We know that 
\[
\frac{m_{G^c}}{{\rm gcd}(m_{F^c},m_{G^c})}e_{m_{F^c}} - 
\frac{m_{G^c}}{{\rm gcd}(m_{F^c},m_{G^c})}e_{m_{G^c}}
\]
is a (possibly non-mimimal) 
generator of ${\rm Syz}_R^1(I)$ by Theorem \ref{generatorsSYZ}.
Moreover, this generator is not a linear first syzygy because
Lemma \ref{mainlemma2} tells us that $|F \cap G| \leq k-1$,
which implies that 
\[\deg \left(\frac{m_{G^c}}{{\rm gcd}(m_{F^c},m_{G^c})}\right) \geq 2 ~~\mbox{and}
\deg \left(\frac{m_{F^c}}{{\rm gcd}(m_{F^c},m_{G^c})} \right)  \geq 2.
\]
To see why, $m_{F^c}$ and $m_{G^c}$ are squarefree monomials, so
\begin{eqnarray*}
\deg ({\rm gcd}(m_{F^c},m_{G^c})) &=& |F^c \cap G^c| = |(F \cup G)^c| 
=  n - |F \cup G| \\
&=& n - |F| - |G| + |F \cap G| \\
&\leq&
 n - (k+1) - (k+1) + (k-1) = n - k -3.
\end{eqnarray*}
Since $\deg (m_{F^c}) = \deg (m_{G^c}) = n - k -1$, the result follows.

Now suppose that ${\rm Syz}_R^1(I)$ is generated by linear first syzygies.
So, in particular there are facets $H_1,\ldots,H_t \in \{F_1,\ldots,F_s\}$,
not necessarily distinct,
so that we can write
\begin{equation}\label{equality}
\frac{m_{G^c}}{{\rm gcd}(m_{F^c},m_{G^c})}e_{m_{F^c}} - 
\frac{m_{F^c}}{{\rm gcd}(m_{F^c},m_{G^c})}e_{m_{G^c}} 
\end{equation}
\[
= \sum_{i=1}^t A_i\left(\frac{m_{H_i^c}}{{\rm gcd}(m_{H_i^c},m_{H_{i+1}^c})}
e_{m_{H_{i+1}^c}} - 
\frac{m_{H_{i+1}^c}}{{\rm gcd}(m_{H_i^c},m_{H_{i+1}^c})}e_{m_{H_i^c}}\right),\]
where each 
$\frac{m_{H_i^c}}{{\rm gcd}(m_{H_i^c},m_{H_{i+1}^c})}
e_{m_{H_{i+1}^c}} - 
\frac{m_{H_{i+1}^c}}{{\rm gcd}(m_{H_i^c},m_{H_{i+1}^c})}e_{m_{H_i^c}}$ is 
a linear first syzygy.  

Note that if the facet $H$ has increment $d$,
the largest possible increment,
and 
\[\frac{m_{H^c}}{{\rm gcd}(m_{H^c},m_{K^c})}
e_{m_{K^c}} - 
\frac{m_{K^c}}{{\rm gcd}(m_{H^c},m_{K^c})}e_{m_{H^c}}\]
is any linear first syzygy involving $H$, then $K$
must also have increment $d$.  Indeed, if the
increment of $K$ is $d'\neq d$, then we could again use
Lemma \ref{mainlemma2} to show that 
\[\deg \left(\frac{m_{H^c}}{{\rm gcd}(m_{H^c},m_{K^c})}\right) \geq 2
~~\mbox{and}~~
\deg \left(\frac{m_{K^c}}{{\rm gcd}(m_{H^c},m_{K^c})}\right) \geq 2,\]
contradicting the fact we have a linear first syzygy.

Because $e_{m_{F^c}}$ appears on both sides of \eqref{equality}, at
least one of the $H_i$s must be $F$.  In the light
of discussion in the previous paragraph, we are forced to have
\[\frac{m_{G^c}}{{\rm gcd}(m_{F^c},m_{G^c})}e_{m_{F^c}} 
= \sum  A_{H,K} \left(\frac{m_{H^c}}{{\rm gcd}(m_{H^c},m_{K^c})}
e_{m_{K^c}} - 
\frac{m_{K^c}}{{\rm gcd}(m_{H^c},m_{K^c})}e_{m_{H^c}} \right),\]
where all the $H$ and $K$ have increment $d$.  That is,
all the linear first syzygies involving a facet 
with increment $d$ must appear together.  
But this means
that 
\[ 0 = \varphi\left(\frac{m_{G^c}}{{\rm gcd}(m_{F^c},m_{G^c})}e_{m_{F^c}}\right) = 
\frac{m_{G^c}}{{\rm gcd}(m_{F^c},m_{G^c})}m_{F^c} \neq 0,\]
which is false.  Here, $\varphi$ is the $R$-module homomorphism used to
define ${\rm Syz}_R^1(I)$.

The proof for $k=2$ is similar.  The only difference
is that  $F$ is picked to be any facet with the largest odd increment,
and we use Lemma \ref{mainlemma1} instead of Lemma \ref{mainlemma2}.
\end{proof}

%%%%%%%%%%%%%%%%%%%%%%%%%%%%%%%%%%%%%%%%%%%%%%%%%%%%%%%%%%%%%%%

\bibliographystyle{plain}

\end{document}